\documentclass[reqno, 12pt, reqno]{amsart}

\usepackage{amsfonts, amsthm, amsmath, amssymb}

\usepackage{hyperref}
\usepackage{a4wide}

\RequirePackage{mathrsfs} \let\mathcal\mathscr

\numberwithin{equation}{section}

\newtheorem{theorem}{Theorem}[section]
\newtheorem{lemma}[theorem]{Lemma}

\theoremstyle{definition}
\newtheorem*{ack}{Acknowledgements}

\renewcommand{\d}{\mathrm{d}}
\renewcommand{\phi}{\varphi}

\newcommand{\0}{\mathbf{0}}

\newcommand{\PP}{\mathbb{P}}

\newcommand{\ZZ}{\mathbb{Z}}

\newcommand{\NN}{\mathbb{N}}
\newcommand{\QQ}{\mathbb{Q}}
\newcommand{\RR}{\mathbb{R}}
\newcommand{\CC}{\mathbb{C}}

\newcommand{\cQ}{\mathcal{Q}}
\newcommand{\cD}{\mathcal{D}}
\newcommand{\cM}{\mathcal{M}}

\renewcommand{\leq}{\leqslant}

\renewcommand{\geq}{\geqslant}

\renewcommand{\bar}{\overline}

\newcommand{\ma}{\mathbf}

\newcommand{\m}{\mathbf{m}}

\newcommand{\x}{\mathbf{x}}
\newcommand{\y}{\mathbf{y}}

\renewcommand{\b}{\mathbf{b}}
\renewcommand{\a}{\mathbf{a}}
\renewcommand{\k}{\mathbf{k}}

\newcommand{\al}{\alpha}
\newcommand{\be}{\beta}
\renewcommand{\rho}{\varrho}

\newcommand{\ve}{\varepsilon}

\DeclareMathOperator{\supp}{supp}

\DeclareMathOperator{\Mod}{mod} 
\renewcommand{\bmod}[1]{\,(\Mod{#1})}

\begin{document}

\title[Pairs of diagonal quadratic forms]{Pairs of diagonal quadratic forms and linear correlations among sums of two squares}

\author{T.D.\ Browning}
\address{School of Mathematics\\
University of Bristol\\ Bristol\\ BS8 1TW\\ United Kingdom}
\email{t.d.browning@bristol.ac.uk}

\author{R.\ Munshi}
\address{School of Mathematics\\ 
Tata Institute of Fundamental Research\\
1 Homi Bhabha Road\\ Colaba\\Mumbai 400005\\ India}
\email{rmunshi@math.tifr.res.in}

\date{\today}

\begin{abstract}
For suitable pairs of diagonal quadratic forms in $8$ variables we use the circle method to investigate the density of simultaneous integer solutions 
and relate this  to the problem of estimating linear correlations among sums of two squares. 

\end{abstract}

\maketitle

\section{Introduction}

Let $Q_1, Q_2\in \ZZ[x_1,\ldots,x_n]$ be 
quadratic forms, with $Q_2$ non-singular.  
Suppose, furthermore,  that as a
variety $V$ in $\PP^{n-1}$, the intersection of quadrics
$Q_{1}=Q_{2}=0$ is also non-singular.
In this paper we return to our recent investigation \cite{BM} into the arithmetic of 
the singular varieties $X\subset \PP^{n+1}$ defined by the pair of quadratic forms
\begin{align*}
q_1(x_1,\dots, x_{n+2})&=Q_1(x_1,\dots,x_n)-x_{n+1}^2-x_{n+2}^2,\\
q_2(x_1,\dots, x_{n+2})&=Q_2(x_1,\dots,x_n).
\end{align*}
Let $r(M)$ be the function that counts the number of representations
of an integer $M$ as a sum of two squares and  
let $W: \RR^{n}\rightarrow \RR_{\geq 0}$ be an infinitely
differentiable bounded function of compact support.
In \cite[Theorem 1]{BM} we were able to prove the expected asymptotic formula for the associated counting function 
\begin{equation}\label{eq:main-sum}
S(B)=\sum_{\substack{\mathbf x \in \mathbb Z^n\\ 2\nmid Q_{1}(\x)\\
Q_2(\x)=0}}r(Q_1(\mathbf
x)) W\left(\frac{\mathbf x}{B}\right),  \quad (B\rightarrow \infty),
\end{equation}
under the assumption that $n\geq 7$. In particular this establishes the Hasse principle for $X$ when $n\geq 7$, 
a fact previously attained in a much more general setting by 
Colliot-Th\'el\`ene, Sansuc and Swinnerton-Dyer \cite{CT}.

Our goal is to show that the sum $S(B)$ can also be estimated asymptotically when $n=6$, provided that  $Q_1$ and $Q_2$ are taken to be diagonal.  We 
will deal here only with forms of the shape
\begin{equation}\label{eq:special}
\begin{split}
Q_1(\x)
&=\al(x_1^2+x_2^2)+\al'(x_3^2+x_4^2),\\
Q_2(\x)
&=\be(x_1^2+x_2^2)+\be'(x_3^2+x_4^2)+\be''(x_5^2+x_6^2),
\end{split}
\end{equation}
where  
$\al,\al',\be,\be',\be''$ are non-zero
integers such that $\al\be'-\al'\be\neq 0$. Note that the common zero locus of 
these polynomials is no longer non-singular in $\PP^5$. 

We will estimate $S(B)$ using the same version of the circle method that we used to handle $n\geq 7$, 
 taking care to avoid duplicating unnecessary effort. We will arrive at the same exponential sums
\begin{align}\label{eq:S'}
S_{d,q}(\m)=\sideset{}{^{*}}\sum_{a\bmod{q}}
\sum_{\substack{\mathbf k \bmod{dq}\\
Q_1(\k)\equiv 0\bmod{d}\\Q_2(\k)\equiv 0\bmod{d}}}
e_{dq}\left(aQ_2(\k)+\m.\k\right),
\end{align}
for positive integers $d$ and $q$ and varying $\m\in \ZZ^{n}$. 
When $Q_1$ and $Q_2$ are both diagonal it will be easier to analyse these sums explicitly.
Nonetheless, the situation for $n=6$ is more delicate, since we are no longer able to
win sufficient cancellation solely through an analysis of 
the Dirichlet series 
$$
\sum_{q=1}^{\infty} \frac{S_{1,q}(\m)}{q^s},
$$
as in \cite{BM}.
Instead we will attempt to profit from cancellation due to
sign changes in the exponential sum $S_{d,1}(\m)$. The latter sum
is associated to a pair of quadratic forms, rather than a single form,
and this raises significant technical obstacles. The following is our main result.

\begin{theorem}
\label{th2}
Assume that  
$Q_1(\x)\gg 1$ and $\nabla Q_1(\x)\gg 1$,
for some absolute implied constant, for every $\x \in
\supp(W)$.
Suppose that $X(\RR)$ and $X(\QQ_{p})$ are non-empty for each prime $p$.  
Then there exists a constant $c>0$ such that 
$$
S(B)=cB^{4}+O(B^{4-\delta}),
$$
for any  $\delta<\frac{1}{160}$.
The  implied constant is allowed to
depend on $\al,\al',\be,\be',\be''$ and $W$.
\end{theorem}

This result compares favourably with work of Cook \cite{C}, who is able to handle suitable pairs of diagonal quadratic forms in at least $9$ variables, rather than the $8$ variables that we deal with.
The leading constant in Theorem~\ref{th2} is an absolutely convergent 
product of local densities
$
c=\sigma_\infty \prod_p \sigma_p,
$
whose positivity is equivalent to the 
hypothesis that  $X(\RR)$ and $X(\QQ_{p})$ are non-empty for each prime $p$.

\medskip

A central problem in analytic number theory is to study the average
order of arithmetic functions as they range over the values taken by
polynomials. 
Let 
$\ma{L}=(L_1,\ldots,L_4)$ be a collection of pairwise non-proportional
binary linear forms defined over $\ZZ$, for which 
each $L_i(x,y)$ is congruent to $x$ modulo $4$ as a polynomial.  
Our choice of forms \eqref{eq:special} is largely motivated by their connection to 
the sums
\begin{equation}
  \label{eq:sumFi-smooth}
T_\omega(B;\ma{L})=\sum_{\substack{(x,y)\in \ZZ^2\\ 2\nmid x}}
r(L_1(x,y))\cdots
r(L_4(x,y))\omega\Big(\frac{x}{B},\frac{y}{B}\Big),
\end{equation}
with   $\omega:\RR^2\rightarrow\RR_{\geq 0}$  a suitable weight function.  
When $\omega=1_{\mathcal{R}}$ is taken to be the characteristic function of an open, bounded and convex region 
$\mathcal{R}\subset \RR^2$, with 
piecewise continuously differentiable boundary, 
it is possible to derive an asymptotic formula for the sum, as $B\rightarrow \infty$. 
This has been the focus of work by Heath-Brown   \cite{h-b03}, which in
turn has been improved in joint work 
of the first author with de la Bret\`eche \cite{4linear}. 
Assume that $L_i(x,y)>0$ for every $(x,y)\in\mathcal{R}$.  Then 
there exists a constant $c$ such that   
\begin{equation}
  \label{eq:4linear}
T_{1_{\mathcal{R}}}(B;\ma{L})=
c B^2 +
O\Big(\frac{B^2}{(\log B)^{ \eta}}\Big), 
\end{equation}
for any $\eta<0.08607$, where the
implied constant is allowed to depend on $L_1,\ldots,L_4, \mathcal{R}$
and $c$
can be interpreted as a product of local densities.
This topic has also been addressed by Matthiesen \cite{lilian} 
using recent developments in additive combinatorics. 
In this case a far-reaching generalisation of 
$T_{1_{\mathcal{R}}}(B;\ma{L})$ is studied, which as a special case retrieves 
the asymptotic formula \eqref{eq:4linear},  but without an explicit error term.

Theorem \ref{th2} can be adapted to study 
$T_\omega(B;\ma{L})$ for other weights $\omega$. 
We make the choice
$$
\omega(x,y)=w_1\big(L_1(x,y)\big)w_0\big(L_2(x,y)\big)w_0\big(L_3(x,y)\big)w_1\big(L_4(x,y)\big),
$$
where $w_0,w_1:\RR\rightarrow \RR_{\geq 0}$ are infinitely
differentiable bounded functions of compact support, with $w_1$  supported away from $0$.
Suppose that 
$
L_i(x,y)=a_ix+b_iy,
$
with $(a_i,b_i)$ congruent to $(1,0)$ modulo ${4}$, 
for $1\leq i\leq 4$. 
For each $1\leq i<j\leq 4$ we  write
$\Delta_{i,j}=a_ib_j-a_jb_i$ for the non-zero
resultant of $L_i$ and $L_j$. For simplicity we will assume that $\Delta_{1,2}=1$, although 
the general case can be handled with more work. 
Opening up the $r$-functions we see that
$$
T_{\omega}(B;\ma{L})=
\sum_{\substack{(x,y)\in \ZZ^2\\ 2\nmid x}}
\sum_{\ma{s},\ma{t}}
\omega\Big(\frac{x}{B},\frac{y}{B}\Big),
$$
where  the inner sum is
over $(\ma{s},\ma{t})\in \ZZ^8$ for which $L_i(x,y)=s_i^2+t_i^2$, for
$1\leq i\leq 4$.
It is clear that the condition $2\nmid x$ is equivalent to $2\nmid s_4^2+t_4^2$ since $L_4(x,y)\equiv x \bmod{4}$.
Eliminating $x,y$ via the transformation 
$$
x=b_2(s_1^2+t_1^2)-b_1(s_2^2+t_2^2),
\quad
y=a_1(s_2^2+t_2^2)-a_2(s_1^2+t_1^2),
$$
we see that the system of four equations is
equivalent to the pair of quadratics
\begin{align*}
\Delta_{2,3}(s_1^2+t_1^2)
-\Delta_{1,3}(s_2^2+t_2^2)+s_3^2+t_3^2&=0,\\
\Delta_{2,4}(s_1^2+t_1^2)
-\Delta_{1,4}(s_2^2+t_2^2)+s_4^2+t_4^2&=0.
\end{align*}
 Since either equation involves a sum of two squares of
variables not apparent in the other equation, this variety is clearly
of the type that are central to the present investigation. 
Taking 
$
W(\x)=w_1(x_1^2+x_2^2)w_0(x_3^2+x_4^2)w_0(x_5^2+x_6^2)w_1(Q_1(\x)),
$
for $\x=(x_1,\ldots,x_6),$ one sees  that 
$
T_{\omega}(B;\ma{L})=S(B^{\frac{1}{2}}),
$
where $Q_1,Q_2$ are as in \eqref{eq:special}, with 
\begin{align*}
(\alpha, \alpha', \beta, \beta',\beta'')=\left(-\Delta_{2,4}, 
\Delta_{1,4},
\Delta_{2,3}, -\Delta_{1,3}, 
1\right).
\end{align*}
The following result is now a trivial consequence of   Theorem \ref{th2}.

\begin{theorem}\label{t:4linear}
Let $\delta<\frac{1}{320}$. 
Then there exists a constant $c$ such that 
$$
T_{\omega}(B;\mathbf{L})=
c B^2 + O(B^{2-\delta}).
$$
\end{theorem}

The constant $c$ 
appearing in Theorem \ref{t:4linear} is a product of local densities. As in Theorem~\ref{th2} one can ensure its positivity by determining whether or not the underlying variety has points everywhere locally. 
At the expense of additional labour it would be possible to 
work with a more general class of weight functions than the one we
have chosen. In this way it seems feasible to  
substantially improve the error term in \eqref{eq:4linear} by 
 selecting a weight
function that approximates the characteristic function of $\mathcal{R}$.

While interesting in their own right,  the 
study of sums like \eqref{eq:sumFi-smooth} can 
play an important r\^ole in the Manin conjecture for
rational surfaces.  This arises from using descent to pass from counting
rational points of bounded height on a surface $S$ to counting suitably
constrained integral points on associated torsors
$\mathcal{T}\rightarrow S$ above the surface. 
The 
asymptotic formula \eqref{eq:4linear} can be interpreted as 
the density of integral points on a torsor above the Ch\^atelet
surface
$$
y^2+z^2=f(x),
$$
with $f$ a totally reducible separable polynomial of degree $3$ or $4$ defined
over $\QQ$. In joint work of the first author with de la
Bret\`eche and Peyre \cite{chat-1}, this is a crucial ingredient in 
the resolution of the Manin conjecture
for this family of  Ch\^atelet
surfaces.  It seems
likely that Theorem~\ref{t:4linear} could prove the basis of
an improved error term in this work.

 \begin{ack}
While working on this paper the first author was 
supported by {\em ERC grant} 306457 and the 
 second author was supported by
{\em SwarnaJayanti Fellowship} 2011--12, DST, Government of India.
\end{ack}

\section{Preliminaries}

Our analysis of $S(B)$ in \eqref{eq:main-sum} is largely based on our previous work \cite{BM}. 
We shall follow the same conventions regarding notation that were introduced there. 
Recall the definition \cite[Eq.\ (3.5)]{BM} of $I_{d,q}(\m)$. We begin by recording a  version of  \cite[Lemma 12]{BM}, in which 
a partial derivative with respect to $d$ is taken. 
 
\begin{lemma}
\label{ubI_q2}
For $0<|\m| \leq dQB^{-1+\varepsilon}=\sqrt{d}B^{\ve}$ and $q\ll Q=B/\sqrt{d}$, we have
\begin{align*}
\frac{\partial^{i}}{\partial d^i}I_{d,q}(\m) &\ll d^{-i}\left|\frac{B\m}{dq}\right|^{1-\frac{n}{2}} B^{\varepsilon},
\end{align*}
for any $i\in\{0,1\}$.
\end{lemma} 
\begin{proof}

When $i=0$ this is due to 
Heath-Brown \cite[Lemma 22]{H}.  
Let us suppose that $i=1$. After a change of variables we have
$$
I_{d,q}(\m)=d^n\int_{\mathbb R^n}
h\left(\frac{q\sqrt{d}}{B},d^2Q_2(\y)\right)W_{d,T}\left(d\y\right) 
e_{4q}(-B\m.\y)\d\mathbf y.
$$
We proceed to take the derivative with respect to $d$. The  right hand side is seen to be 
\begin{align*}
\frac{n}{d}I_{d,q}(\m)+d^n\int_{\mathbb R^n}
g_d(\y)  e_{4q}(-B\m.\y)\d\mathbf y,
\end{align*}
where if $h^{(1)}(x,y)=\frac{\partial}{\partial x}h(x,y)$ and $h^{(2)}(x,y)=\frac{\partial}{\partial y}h(x,y)$, then
\begin{align*}
g_d(\y)=~&
\frac{q}{2B\sqrt{d}}h^{(1)}\left(\frac{q\sqrt{d}}{B},d^2Q_2(\y)\right)W_{d,T}\left(d\y\right)\\ 
&
+2dQ_2(\y)h^{(2)}\left(\frac{q\sqrt{d}}{B},d^2Q_2(\y)\right)W_{d,T}\left(d\y\right) 
+h\left(\frac{q\sqrt{d}}{B},d^2Q_2(\y)\right)\frac{\partial}{\partial d} W_{d,T}\left(d\y\right). 
\end{align*}
Let $W^{(1)}(\y)=\y.\nabla W(\y)$. 
One finds that 
$$
\frac{\partial}{\partial d} W_{d,T}\left(d\y\right)= 
 \frac{1}{d}W^{(1)}\left(d\mathbf y\right)V_T(d)+W\left(d\mathbf y\right)V_T'(d),
 $$
if $T\leq B$, and 
$$
\frac{\partial}{\partial d} W_{d,T}\left(d\y\right) 
=\frac{1}{d}W^{(1)}\left(d\mathbf y\right)V_T\left(B^2dQ_1(\y)\right)+W\left(d\mathbf y\right)V_T'\left(B^2dQ_1(\y)\right)B^2Q_1(\y),
$$
otherwise.
Hence
\begin{align*}
\frac{\partial}{\partial d} W_{d,T}\left(d\y\right)=\frac{1}{d}\hat W_{d,T}\left(d\y\right), 
\end{align*}
where the new function $\hat W_{d,T}$ has the same analytic behaviour as $W_{d,T}$. Another change of variables now yields
\begin{align*}
\frac{\partial}{\partial d}I_{d,q}(\m)
=~&\frac{n}{d}I_{d,q}(\m)+\frac{1}{2d}\int_{\mathbb R^n}\frac{q\sqrt{d}}{B}h^{(1)}\left(\frac{q\sqrt{d}}{B},Q_2(\y)\right)W_{d,T}\left(\y\right)e_{4dq}(-B\m.\y)\d\mathbf y\\
&+\frac{2}{d}\int_{\mathbb R^n}h^{(2)}\left(\frac{q\sqrt{d}}{B},Q_2(\y)\right)\bar W_{d,T}\left(\y\right)e_{4dq}(-B\m.\y)\d\mathbf y\\
&+\frac{1}{d}\int_{\mathbb R^n}h\left(\frac{q\sqrt{d}}{B},Q_2(\y)\right)\hat W_{d,T}\left(\y\right)e_{4dq}(-B\m.\y)\d\mathbf y,
\end{align*}
where $\bar W_{d,T}(\y)=W_{d,T}(\y)Q_2(\y)$. The last three integrals can be compared with $I_{d,q}(\m)$, and the lemma now follows using the bounds in the statement of the lemma for $i=0$.
\end{proof}

Let  $\rho(d)=S_{d,1}(\mathbf{0})$, in the notation of \eqref{eq:S'}.
In \cite[\S 1]{BM} we defined ``Hypothesis-$\rho$'' to be the hypothesis that 
$
\rho(d)=O(d^{n-2+\ve}),
$
for any $\ve>0$. Our present investigation will be streamlined substantially by the convention adopted in \cite{BM} that 
any estimate concerning  quadratic forms 
$Q_1,Q_2\in \ZZ[x_1,\ldots,x_n]$ 
was valid for arbitrary forms such that  $Q_2$ is non-singular, with $n\geq 5$, for which the variety $Q_1=Q_2=0$ defines a (possibly singular) geometrically integral complete intersection
$V\subset \PP^{n-1}$. 
The quadratic forms $Q_1,Q_2$ in \eqref{eq:special} clearly adhere to these constraints. 
Our next task is to verify  Hypothesis-$\rho$ in the present setting.

\begin{lemma}
\label{rho(d)-special}
Hypothesis-$\rho$ holds if  $Q_1,Q_2$ are given by 
\eqref{eq:special}.
\end{lemma}

\begin{proof}
By multiplicativity, it suffices to analyse the case $d=p^r$.
Note that for given $u\bmod{p^r}$, the number of $x_1,x_2 \bmod{p^r}$ such that $x_1^2+x_2^2\equiv u \bmod{p^r}$ is at most  $(1+r)p^r$. It follows that 
$$
\rho(p^r)\leq
(1+r)^3p^{3r}\#\{(u,v,w) \bmod{p^r}: p^r\mid \al u+\al 'v, ~p^r\mid \be u+\be'v+\be''w\}.
$$ 
Suppose $p^k\| \be''$.  We may clearly assume without loss of generality  that $ r>k$. Then from the congruence $p^r\mid \be u+\be'v+ \be''w$ we get a
congruence modulo $p^{r-k}$ which gives a unique solution for $w$ modulo
$p^{r-k}$. These lift  to give us at most $ p^k$ possibilities for $w$ modulo $p^r$, for any given $u,v$. Similarly  the congruence $p^r\mid \al u+\al'v$ gives at most $p^j$
many $u$ for any given $v$, where $p^j\| \alpha$.
Hence $\rho(p^r)\ll (1+r)^3p^{4r}$, which is satisfactory for the lemma.
\end{proof}

The exponential sum $S_{d,q}(\m)$ in \eqref{eq:S'} satisfies the  multiplicativity property recorded in   \cite[Lemma 10]{BM}. This makes it natural to introduce the sums 
$$
\mathcal
Q_{q}(\m)=S_{1,q}(\m),\quad
\mathcal D_{d}(\m)=S_{d,1}(\m),\quad 
 \mathcal
M_{d,q}(\m)=S_{d,q}(\m),
$$ 
the latter sum only being of interest when $d$ and $q$ exceed $1$ and
are constructed from the same set of primes. 
Since the variety $V$ defined by the common zero locus of
$Q_1$ and $Q_2$ is 
singular, we will need  alternatives to the
estimates obtained in \cite[\S 5]{BM} for 
$\cD_d(\m)$.

In this section, using \cite{BM},  we shall establish
the veracity of Theorem \ref{th2} subject to new bounds for the 
exponential sums $\cD_d(\m)$ and $\cM_{d,q}(\m)$,
whose truth will be demonstrated in subsequent sections.
We can be completely explicit about the analogue of the polynomial
$G(\m)$ in \cite[\S 5]{BM}.
Define 
\begin{equation}\begin{split} \label{eq:ci}
c_{0}&=c_{0}(\m)= \al \al' (m_{5}^{2}+m_{6}^{2}),\\
c_{1}&=c_{1}(\m)= \al'\be'' (m_{1}^{2}+m_{2}^{2}) +\al\be'' (m_{3}^{2}+m_{4}^{2}) + 
(\al\be'+ \al'\be) (m_{5}^{2}+m_{6}^{2}),\\
c_{2}&=c_{2}(\m)= \be' \be'' (m_{1}^{2}+m_{2}^{2}) + \be \be''(m_{3}^{2}+m_{4}^{2}) + 
\be\be' (m_{5}^{2}+m_{6}^{2}).
\end{split}
\end{equation}
In particular  $c_2=Q_2^*(\m)$, where $Q_2^*$ is the adjoint quadratic form.
We will set 
\begin{equation}\label{eq:delta-given}
\delta(\m)=c_{1}^{2}-4c_{0}c_{2},
\end{equation}
a quartic form in $\m$,
and 
\begin{equation}\label{eq:sigma}
\sigma(\m)=
\alpha'\beta''(m_1^2+m_2^2)+
 \alpha\beta'' (m_3^2+m_4^2)+(\alpha\beta'-\alpha'\beta)(m_5^2+m_6^2).
 \end{equation}
The r\^ole of $G$ is now  played by the polynomial
$\delta(\m)H(\m)$, where 
$$
H(\m)=(m_1^2+m_2^2)(m_3^2+m_4^2)(m_5^2+m_6^2).
$$
We henceforth put
$\Delta_V=2\al\al' \be \be'\be'' (\al\be'-\al'\be)\neq 0$.

Our proof of \cite[Theorem 1]{BM} was based on a careful analysis of the
sum $U_{T,\ma{a}}(B,D)$ in  \cite[Eq.\ (7.3)]{BM},  for  $D\geq 1$. 
Rather than summing non-trivially over $q$, as there, our course of
action for Theorem~\ref{th2} is based on summing non-trivially over
$d$. As before it suffices to consider the contribution to $ 
S_{T,\a}^\sharp(B)$  from $\m$ such that $\m=\ma{0}$ or 
$0<|\m|\leq \sqrt{d}B^{\varepsilon}$. 
Dealing with the latter contribution leads us to study the expression
$$
V_{T,\a}(B,D)=B^4
\sum_{0<|\m|\leq\sqrt{D}B^{\varepsilon}}
\sum_{q\ll B/\sqrt{D}}\frac{1}{q^6}\left|\sum_{\substack{d\sim D\\
(d,\Delta_V^\infty)\leq \Xi} 
}\frac{\chi(d)}{d^5}
T_{d,q}(\m)I_{d,q}(\m)\right|,
$$
for $D\geq 1$, where $T_{d,q}(\m)$ is given in  \cite[Lemma 8]{BM}.
We will 
show that 
\begin{equation}\label{eq:euro}
V_{T,\a}(B,D)=O(\Xi^{\frac{3}{2}} B^{4-\frac{1}{16}+\ve}),
\end{equation}
for any $D\ll B$.

Define the non-zero integer
\begin{equation}\label{eq:N_T2}
N=\begin{cases}
 \delta(\m)H(\m), & \mbox{if $\delta(\m)H(\m)\neq 0$,}\\
H(\m), & \mbox{if $\delta(\m)=0$ and $H(\m)\neq 0$,}\\
\delta(\m), & \mbox{if $\delta(\m)\neq 0$ and $H(\m)=0$,}\\
1, & \mbox{otherwise.}
\end{cases}
\end{equation}
We split $d$ as
$\delta d$ with $(d,q\Delta_V N)=1$ and $\delta\mid (q\Delta_VN)^{\infty}$. 
Then 
$$
V_{T,\a}(B,D)\leq B^4
\sum_{0<|\m|\leq\sqrt{D}B^{\varepsilon}}
\sum_{q\ll B/\sqrt{D}}\frac{1}{q^6}
\sideset{}{'}\sum_{\substack{\delta \mid (q\Delta_V N)^\infty\\  \delta\leq D\\
(\delta,\Delta_V^\infty)\leq \Xi 
}}
\hspace{-0.2cm}
\frac{|T_{\delta,q}(\m)|}{\delta^5}
\left|\sum_{\substack{d\sim D/\delta \\ (d,q\Delta_V N)=1}}
\hspace{-0.2cm}
\frac{\chi(d)}{d^5}
\cD_{d}(\m)I_{\delta d,q}(\m)\right|,
$$
where $\sum'$ means that the sum is restricted to odd integers only. 
Applying partial summation we see that the inner sum over $d$ can be
written
$$
J=\Sigma(D/\delta) \cdot \frac{I_{D,q}(\m)}{(D/\delta)^5}- \int_{D/(2\delta)}^{D/\delta} \Sigma(x) 
\frac{\partial}{\partial x}\left(\frac{I_{\delta x,q}(\m)}{x^5}\right)\d x,
$$
where
\begin{equation}
  \label{eq:need1-pre}
\Sigma(x)=\sum_{\substack{D/(2\delta) < d\leq x\\ (d,q\Delta_V N)=1}}\chi(d)\cD_{d}(\m).
\end{equation}
We will establish the following result in \S \ref{s:last}.

\begin{lemma}\label{lem:need1}
We have 
$
\Sigma(x)\ll 
|\m|^{\theta(\m)+\ve} x^{\frac{5}{2}+\psi(\m)+\ve}, 
$
with 
$$
\psi(\m)
=\begin{cases}
0, & \mbox{if $\delta(\m)\neq \square$ and $H(\m)\neq 0$},\\
\frac{1}{2}, & \mbox{if $\delta(\m)=\square$ and $H(\m)\neq 0$},\\
\frac{1}{2}, & \mbox{if $\delta(\m)\neq 0$ and $H(\m)= 0$},\\
\frac{3}{2},  & \mbox{otherwise}, 
\end{cases}
$$
and 
$$
\theta(\m)
=\begin{cases}
\frac{7}{8}, & \mbox{if $\delta(\m)\neq \square$ and $H(\m)\neq 0$},\\
0,  & \mbox{otherwise}. 
\end{cases}
$$
\end{lemma}

Taking Lemma \ref{lem:need1} on faith for the moment, and appealing to
Lemma \ref{ubI_q2}, we therefore deduce that
\begin{align*}
J
&\ll B^\ve\left| \frac{B\m }{Dq}\right|^{-2}  
|\m|^{\theta(\m)} \left(\frac{D}{\delta}\right)^{\frac{5}{2}+\psi(\m)}  \left(
  \frac{\delta}{D}\right)^5.
\end{align*}
Inserting this into our expression for $V_{T,\a}(B,D)$ gives
$$
V_{T,\a}(B,D)\ll \frac{B^{2+\ve}}{D^{\frac{1}{2}}}
\sum_{0<|\m|\leq\sqrt{D}B^{\varepsilon}}\frac{|\m|^{\theta(\m)} D^{\psi(\m)}}{|\m|^2}
\sum_{q\ll B/\sqrt{D}}\frac{1}{q^4}
\sideset{}{'}\sum_{\substack{\delta \mid (q\Delta_V N)^\infty\\  \delta\leq D\\
(\delta,\Delta_V^\infty)\leq \Xi 
}}
\frac{|T_{\delta,q}(\m)|}{\delta^{\frac{5}{2}+\psi(\m)}}.
$$
Combining \cite[Lemma 9]{BM} with  \cite[Eq.\ (3.4)]{BM}, we see that
$$
\frac{|T_{\delta,q}(\m)|}{q^4} \ll \frac{|S_{\delta,q'}(\m)|}{{q'}^4},
$$
where $q'$ is the odd part of $q$. Hence we may restrict attention to
odd values of $q$ in the above estimate without loss of generality. 
Let us write $\delta=\delta_1\delta_2$ with $\delta_1\mid \Delta_V^\infty$ and $(\delta_2,\Delta_V)=1$. Similarly we write 
 $q=q_1q_2$ with $q_1\mid \Delta_V^\infty$ and $(q_2,\Delta_V)=1$. In particular $q_2$ is odd
 and we have 
$S_{\delta,q}(\m)=S_{\delta_1,q_1}(\m)S_{\delta_2,q_2}(\m)$ by \cite[Lemma  10]{BM}.
Combining Lemma~\ref{rho(d)-special}, with \cite[Eq.\ (4.1)]{BM} and  \cite[Lemma 25]{BM} we see that 
$S_{\delta_1,q_1}(\m)\ll \delta_1^{4+\ve}q_1^4$, whence
\begin{equation}\label{eq:goat}
\frac{S_{\delta_1,q_1}(\m)}{\delta_1^{\frac{5}{2}} q_1^4}\ll \Xi^{\frac{3}{2}} B^\ve.
\end{equation}
Note that there are $O(B^\ve)$ choices for $\delta_1$ and $q_1$ by \cite[Eq.\ (1.3)]{BM}.
Finally, appealing once more to \cite[Lemma 10]{BM}, we deduce there is a factorisation $\delta_2=\delta_{21}\delta_{22}$ and $q_2=q_{21}q_{22}$, with 
$$
\delta_{21}\mid N^\infty, \quad (\delta_{21},q_{21})=1, \quad \delta_{22}\mid q_{22}^\infty, \quad
(\delta_{21}q_{21},\delta_{22}q_{22})=1,
$$
such that 
\begin{equation}\label{eq:bridge}
S_{\delta_2,q_2}(\m)= \cD_{\delta_{21}}(\m)\cQ_{q_{21}}(\m)\cM_{\delta_{22},q_{22}}(\m).
\end{equation}
We have $(\delta_{21}\delta_{22}q_{21}q_{22},\Delta_V)=1$ here. We will need
good upper bounds for  these sums.

\begin{lemma}\label{lem:need2}
Assume that $(d,\Delta_V)=1$. Then we have 
$$
|\cD_{d}(\m)|
\leq 4^{\omega(d)}\tau(d)^2d^{2}
(d,\m)
(d, \delta(\m)).
$$
\end{lemma}

The proof of this result is deferred to \S \ref{s:last}.
It implies that 
$\cD_{d}(\m)\ll 
 d^{3+\ve} 
(d,\m),
$
which recovers \cite[Lemma 22]{BM} in the present setting. 
An application of \cite[Lemma 25]{BM} yields 
$\cM_{d,q}(\m) \ll
d^{3+\ve} q^{4}$.
This is not fit for purpose when $\m$ is generic, although it does suffice for non-generic $\m$. The 
following result will be established in \S \ref{s:last'}.

\begin{lemma}\label{lem:need3}
Assume that $(d,\Delta_V)=1$. Then we have 
\begin{align*}
\cM_{d,q}(\m)\ll
 d^{2+\ve}q^{3+\ve}
 (d,\m)(q,\m)^2
(d,\delta(\m)) (q,Q_2^*(\m)).
\end{align*}
\end{lemma}

Returning to \eqref{eq:bridge}, 
we are now ready to deduce the estimate
\begin{equation}\label{eq:squash}
S_{\delta_2,q_2}(\m) \ll 
 \delta_2^{2+\ve}q_2^{3+\ve}
 (\delta_2,\m)(q_2,\m)^2
(\delta_2,\delta(\m)) (q_2,Q_2^*(\m)),
\end{equation}
if $\delta(\m)H(\m)Q_2^*(\m)\neq 0$, and 
\begin{equation}\label{eq:for-hyp}
S_{\delta_2,q_2}(\m) \ll 
\delta_2^{3+\ve}q_2^{4}(\delta_2,\m),
\end{equation}
in general. 
The latter follows easily from combining Lemma \ref{lem:need2} with \cite[Eq.\ (4.1)]{BM} and  \cite[Lemma 25]{BM}.
For the former  we act similarly but substitute Lemma \ref{lem:need3} for \cite[Lemma 25]{BM} and \cite[Lemma 15]{BM}
 for  \cite[Eq.\ (4.1)]{BM}.

We proceed to partition 
$\ZZ^6$ into a disjoint union of four sets. 
Let 
$\cM_1$ denote the set of $\m\in \ZZ^6$ such that 
$\delta(\m)\neq \square$ and $H(\m)Q_2^*(\m)\neq 0$. 
Likewise, let 
$\cM_2$ denote the set of $\m$ for which  
$\delta(\m)=\square$ or $Q_2^*(\m)=0$ but $H(\m)\neq 0$. 
Let $\cM_3$  
 denote the set of  $\m$ such that 
$\delta(\m)Q_2^*(\m)\neq 0$ and $H(\m)= 0$.
Finally,  let  $\cM_4$ be the set of $\m$ such that 
$\delta(\m)Q_2^*(\m)=H(\m)= 0$. We will need the following result.

\begin{lemma}\label{lem:bad-m}
Let $M\geq 1$ and let $\ve>0$. Then we have
$$
\#\{
\m\in \cM_i: |\m|\leq M\} =
\begin{cases}
O(M^{4+\ve}),
 & \mbox{if $i=2$ or $3$},\\
O(M^{2+\ve}),
 & \mbox{if $i=4$}.
 \end{cases}
$$
Furthermore, for any $A\in \ZZ$, we have 
$$
\#
\{\m\in \ZZ^6: |\m|\leq M, ~ \delta(\m)=A\} =O((1+|A|)^\ve M^{2+\ve}).
$$
\end{lemma}

\begin{proof}
Let us write $R_i(M)$ for the quantity on the  left hand side in the first displayed equation and $R(A;M)$ for the quantity in the second displayed equation.
Note that $H(\m)=0$ if and only if $m_i=m_j=0$ for some $(i,j)\in \{(1,2),(3,4),(5,6)\}$.
In particular the bound for $R_3(M)$ is trivial. 
Likewise it is easy to see that there are $O(M^{4+\ve})$ choices of $|\m|\leq M$ for which $Q_2^*(\m)=0$ and 
$O(M^{2+\ve})$ choices of $|\m|\leq M$ for which $Q_2^*(\m)=H(\m)=0$.

To handle the contributions  $\delta(\m)$, it will be convenient to make the change of variables
$X=c_1(\m)$, $Y=c_0(\m)$ and $Z=c_2(\m)$, in the notation of \eqref{eq:ci}. 
In particular we have $X,Y,Z\ll M^2$.
Using the familiar estimate $r(m)=O(m^\ve)$
and recalling that $\Delta_V\neq 0$, 
we see that there are $O(M^\ve)$ choices of $\m$ associated to a given triple  $X,Y,Z$. 
We begin by estimating 
$$
R(A;M)
\ll M^\ve
\#\{X,Y,Z\ll M^2: X^2-4YZ=A\}.
$$
For a fixed $X$ with  $X^2-A\neq 0$, the trivial estimate for the divisor function reveals that there are $O((1+|A|)^\ve M^\ve)$ choices for $Y,Z$. For $X$  with $X^2-A=0$  there are $O(M^2)$ choices for $Y,Z$ such that $YZ=0$. Combining these contributions  therefore shows that $R(A;M)=O((1+|A|)^\ve M^{2+\ve})$, as claimed.

Turning to the estimation of $R_2(M)$ and $R_4(M)$, we first observe that
$$
R_2(M)\ll M^{4+\ve}+M^\ve\#\{W,X,Y,Z\ll M^2: X^2-4YZ=W^2\} \ll M^{4+\ve}.
$$
Finally, we  have 
$
R_4(M)\ll M^{2+\ve}+M^\ve R(0,M)\ll  M^{2+\ve},
$
which   completes the proof.
\end{proof}

Let us write $V_i$ for the overall contribution to 
$V_{T,\a}(B,D)$ from $\m\in \cM_i$, for $1\leq i\leq 4$.
We relabel $\delta_2$ to be $\delta$ and $q_2$ to be $q$.
Beginning with the contribution from generic $\m$, an application of \eqref{eq:goat} and \eqref{eq:squash} yields
\begin{align*}
V_1
\ll~& \frac{\Xi^{\frac{3}{2}} B^{2+\ve}}{D^{\frac{1}{2}}} 
\sum_{\substack{
\m\in \cM_1\\
|\m|\leq\sqrt{D}B^{\varepsilon}}}
|\m|^{-\frac{9}{8}} 
\sum_{q\ll B/\sqrt{D}}\frac{(q,\m)^2(q,Q_2^*(\m))}{q}
\sum_{\substack{\delta \mid (qN)^\infty\\  \delta\leq D}}
\frac{(\delta,\m)  (\delta,\delta(\m))}{\delta^{\frac{1}{2}}}.
\end{align*}
The inner sum is 
\begin{align*}
\sum_{\substack{\delta \mid (qN)^\infty\\  \delta\leq D}}
\frac{(\delta,\m)  (\delta,\delta(\m))}{\delta^{\frac{1}{2}}}
&\ll B^\ve \max_{\delta\leq D} \left\{(\delta,\m)  (\delta,\delta(\m))^{\frac{1}{2}}\right\}.
\end{align*}
Extracting  the greatest common divisor $h$ of $m_1,\ldots,m_6$ and rearranging our expression, we therefore obtain
\begin{align*}
V_1
\ll~& \frac{\Xi^{\frac{3}{2}} B^{2+\ve}}{D^{\frac{1}{2}}} 
\max_{\delta\leq D} 
\sum_{h\leq \sqrt{D}B^{\varepsilon}} 
\hspace{-0.1cm}
\frac{(\delta,h)}{h^{\frac{9}{8}}}
\hspace{-0.2cm}
\sum_{\substack{
\m\in \cM_1\\
|\m|\leq h^{-1}\sqrt{D}B^{\varepsilon}}}
\hspace{-0.3cm}
\frac{(\delta,\delta(h\m))^{\frac{1}{2}}}{|\m|^{\frac{9}{8}} 
}
\sum_{q\ll B/\sqrt{D}}
\hspace{-0.3cm}
\frac{(q,h)^2(q,Q_2^*(h\m))}{q}.
\end{align*}
The inner sum over $q$ is clearly $O(h^2
B^\ve)$, whence
$$
V_1
\ll \frac{\Xi^{\frac{3}{2}} B^{2+\ve}}{D^{\frac{1}{2}}} 
\max_{\delta\leq D} 
\sum_{h\leq \sqrt{D}B^{\varepsilon}} 
\hspace{-0.1cm}
\frac{h^2(\delta,h)^{3}}{h^{\frac{9}{8}}}
\sum_{\substack{
\m\in \cM_1\\
|\m|\leq h^{-1}\sqrt{D}B^{\varepsilon}}}
\frac{(\delta,\delta(\m))^{\frac{1}{2}}}{|\m|^{\frac{9}{8}} 
},
$$
since $(\delta,\delta(h\m))\leq (\delta,h)^4(\delta,\delta(\m))$.

For a fixed integer $\delta\leq D$ and a choice of $M$ with
$0<M\leq h^{-1}\sqrt{D}B^{\ve}$, we deduce from Lemma 
\ref{lem:bad-m} that 
\begin{align*}
\sum_{\substack{M<|\m|\leq 2M\\
\delta(\m)\neq 0}}
(\delta,\delta(\m))
&\leq 
\sum_{0<|A|\ll M^4} 
(\delta, A)
R(A;M)\\
&\ll M^{2+\ve}
\sum_{0<|A|\ll M^4} 
(\delta, A)\\
&\ll \delta^\ve M^{6+\ve}.
\end{align*}
Armed with this we conclude that 
\begin{align*}
V_1
&\ll \frac{\Xi^{\frac{3}{2}} B^{2+\ve}}{D^{\frac{1}{2}}} 
\max_{\delta\leq D}
\sum_{h\leq \sqrt{D}B^{\varepsilon}} 
\frac{h^2(\delta,h)^{3}}{h^{\frac{9}{8}}}
\left(\frac{\sqrt{D}}{h}\right)^{6-\frac{9}{8}}
\ll \Xi^{\frac{3}{2}} B^{2+\ve} D^{2-\frac{1}{16}+\ve}.
\end{align*}
This is  $O(\Xi^{\frac{3}{2}} B^{4-\frac{1}{16}+\ve})$, since 
$D\ll B$, which is satisfactory for \eqref{eq:euro}.

It is now time to consider the contribution from  non-generic $\m$.
Invoking our  estimate for $S_{\delta,q}(\m)$ in \eqref{eq:for-hyp}
we find that 
\begin{align*}
V_i
\ll \frac{\Xi^{\frac{3}{2}} B^{3+\ve}}{D} 
\max_{\delta\leq D}
\sum_{\substack{
\m\in \cM_i\\
0<|\m|\leq\sqrt{D}B^{\varepsilon}}}
|\m|^{-2} D^{\psi(\m)}
\frac{\delta^{\frac{1}{2}}(\delta,\m)}{\delta^{\psi(\m)}},
\end{align*}
for $2\leq i\leq 4$.
Suppose first that $i\in \{2,3\}$, so that $\psi(\m)\leq \frac{1}{2}$. Then 
Lemma \ref{lem:bad-m} implies that
\begin{align*}
V_i
&\ll \frac{\Xi^{\frac{3}{2}} B^{3+\ve}}{D^{\frac{1}{2}}} 
\max_{\delta\leq D}
\sum_{\substack{\m\in \cM_i\\
|\m|\leq\sqrt{D}B^{\varepsilon}}}
|\m|^{-2} 
(\delta,\m)\\
&\ll
\frac{\Xi^{\frac{3}{2}} B^{3+\ve}}{D^{\frac{1}{2}}} 
\max_{\delta\leq D}
 \sum_{\ell \mid \delta} \ell
\sum_{\substack{\m\in \cM_2\\
|\m|\leq\sqrt{D}B^{\varepsilon}\\
\ell\mid \m
}}
|\m|^{-2}\\
&\ll \frac{\Xi^{\frac{3}{2}} B^{3+\ve}}{D^{\frac{1}{2}}}  
\left(\sqrt{D}B^\ve\right)^{2+\ve}\\
&\ll  \Xi^{\frac{3}{2}} B^{4-\frac{1}{2}+\ve},
\end{align*}
since $D\ll B$.
This is 
 satisfactory for \eqref{eq:euro}.

Finally, when $i=4$, so that $\psi(\m)\leq \frac{3}{2}$,
Lemma \ref{lem:bad-m} yields
\begin{align*}
V_4 &\ll \Xi^{\frac{3}{2}} B^{3+\ve} D^{\frac{1}{2}} 
\sum_{\substack{\m\in \cM_4\\
|\m|\leq\sqrt{D}B^{\varepsilon}}}
|\m|^{-2} 
\ll 
\Xi^{\frac{3}{2}} B^{4-\frac{1}{2}+\ve},
\end{align*}
since $D\ll B$.  This too is satisfactory for \eqref{eq:euro} and thereby 
concludes our treatment of the contribution from non-zero $\m$ 
to $S_{T,\a}^\sharp(B)$.

Our analogue of \cite[Eq.\ (8.1)]{BM} is now
$$
S(B)= 
M^\sharp(B)+
O(\Xi^{-\frac{1}{6}}B^{4+\ve}+\Xi B^{3+\ve}
+
\Xi^{\frac{3}{2}} B^{4-\frac{1}{16}+\ve}),
$$
where $M^\sharp (B)$ is given by 
\cite[Eq.\ (8.2)]{BM} with $n=6$.
To handle the remaining term we may  invoke \cite[Lemma 30]{BM}.
Taking $\Xi=B^{\frac{3}{80}}$ we obtain the final error term $O(B^{4-\frac{1}{160}+\ve})$ in our asymptotic formula for $S(B)$.
This therefore completes the proof of 
Theorem \ref{th2}, subject to the exponential sum estimates Lemmas \ref{lem:need1}--\ref{lem:need3}.

\section{Analysis of $\mathcal{D}_d(\m)$}\label{s:last}

In this section we  establish Lemmas \ref{lem:need1} and \ref{lem:need2},  for which we will need to undertake a detailed analysis of the sum $\cD_{d}(\m)$  when  $Q_{1}, Q_{2}$ are given by 
\eqref{eq:special} and $(d,\Delta_V)=1$, with 
$\Delta_{V}=2\al \al' \be \be' \be'' (\al\be'-\al'\be)$.
Before launching into this endeavour let us
 record  a  preliminary result concerning the quantity 
$$
\rho_{f}(p^r)=\#\{ 
n\bmod{p^r}: f(n)\equiv 0\bmod{p^r}\},
$$
for any quadratic polynomial $f$ defined over $\ZZ$ and any prime power $p^r$.

\begin{lemma}\label{lem:rho}
Let $r\geq 1$ and let $f(x)=c_{0}x^{2}+c_{1}x+c_{2}$ for $c_{0},c_{1},c_{2}\in \ZZ$.
Then we have
$$
\rho_{f}(p^{r})\leq 2p^{\frac{v_p(c_{1}^{2}-4c_{0}c_{2})}{2}}.
$$
\end{lemma}

\begin{proof}
Let $\mu=v_p(c_{1}^{2}-4c_{0}c_{2})$ and let  $p^\ell\| (c_{0},c_{1},c_{2})$.
Writing $c_i=p^\ell c_i'$, for $0\leq i\leq 2$, we
denote by $f'$ the quadratic polynomial with coefficients $c_0',c_1',c_2'$.  
In particular we have  $\mu'=v_p(c_1'^2-4c_0'c_2')=\mu-2\ell$.
If $r\leq \ell$ then it is clear that $\rho_f(p^r)=p^r\leq p^{ \frac{\mu}{2}}$, which  is  satisfactory.
Alternatively, if $r>\ell$, then
$\rho_f(p^r)=p^\ell \rho_{f'}(p^{r-\ell}).$
The content of $f'$ is now coprime to $p$ and it therefore follows from work   of 
Huxley \cite{hux} that
$
\rho_{f'}(p^{r-\ell}) \leq 2p^{\frac{\mu'}{2}}.
$
But then 
$
\rho_f (p^r)
\leq 2 p^{\frac{\mu}{2}}
$
when  $r>\ell$, as required to  complete the proof of the lemma.
\end{proof}

We base  our analysis of $\cD_{d}(\m)$
 on the initial steps after the statement of \cite[Lemma 20]{BM}. 
Let 
\begin{equation}\label{eq:psy}
\cD_d(\m;\b) = \sum_{\k\bmod{d}} e_d\left(b_1Q_1(\k)+b_2Q_2(\k)+\m.\k\right).
\end{equation}
Extracting the  greatest common divisor
 between $d$ and $\b$, as there, we conclude that 
\begin{equation}
  \label{eq:seek}
\begin{split}
\mathcal D_{d}(\ma{m})&=
\sum_{\substack{h\mid (d, \m)}} h^4 \cdot \frac{1}{d'^2}
\hspace{0.2cm}
\sideset{}{^{*}}\sum_{\b\bmod{d'}} 
\mathcal D_{d'}(\ma{m}';\b)
=
\sum_{\substack{h\mid (d,  \m)}} h^4 \cdot 
\mathcal D_{d'}^*(\ma{m}'),
\end{split}
\end{equation}
say, with $d=hd'$ and $\m=h\m'$. 
By  multiplicativity  it suffices to analyse 
$\mathcal D_{d'}^*(\ma{m}')$  when $d'=p^{r}$ for $r\geq 1$ and  $p\nmid \Delta_{V}$.  Let 
\begin{equation}\label{eq:LL'L''}
\begin{split}
L(x,y)&=L_{1}(x,y)=L_{2}(x,y)=\al x+ \be y, \\
L'(x,y)&=L_{3}(x,y)=L_{4}(x,y)=\al' x+ \be' y, \\
L''(x,y)&=L_{5}(x,y)=L_{6}(x,y)= \be'' y.
\end{split}
\end{equation}
We will write $g(x,y)=L(x,y)L'(x,y)L''(x,y)$. 
Then we are led to consider
$$
\cD_{p^{r}}(\m';\b) = \prod_{i=1}^{6} \sum_{k\bmod{p^{r}}}
e_{p^{r}}\left( L_{i}(\b)k^{2}+m_{i}'
k\right).
$$
We would like to employ the explicit formulae for  Gauss sums to evaluate these sums, which we proceed to recall. 
Let $p$ be a prime with $p\nmid 2a$. Then we have 
\begin{equation}\label{eq:gauss}
\sum_{k \bmod{p^r}}e_{p^r}(a k^2+m k)=p^{\frac{r}{2}} e_{p^r}(-\overline{4a}m^2)
\times \begin{cases}1, &\mbox{if $r$ is even,}\\
\chi_p(a)\varepsilon(p), &\mbox{if $r$ is odd.}
\end{cases}
\end{equation}
Here $\chi_p(\cdot)=(\frac{\cdot}{p})$ is the Legendre symbol and $\ve(p)=1$ or $i$ according to whether $p$ is congruent to $1$ or $3$ modulo $4$, respectively. We will also need the following evaluation of the Ramanujan sum
\begin{equation}\label{eq:ram}
c_{p^r}(b)=
\sideset{}{^{*}}\sum_{k\bmod{p^r}} e_{p^r}(bk) =\begin{cases}
\phi(p^r), & \mbox{if $p^r\mid b$,}\\
-p^{r-1}, & \mbox{if $p^{r-1}\| b$,}\\
0, & \mbox{otherwise}.
\end{cases}
\end{equation}
Armed with these facts we be able to prove the following result.

\begin{lemma}\label{lem:mawkish}
Let $p\nmid \Delta_V$ and let $\m \in \ZZ^6$. Then we have
$$
\mathcal{D}_{p^r}^*(\m)
\leq 4(r+1)p^{2r+\min\{r,\frac{v_p(\delta(\m))}{2}\}}.
$$
Moreover, when $r=1$ and $p\nmid H(\m)$, we have 
$$
\cD_{p}^*(\m)=p^2 \chi_{p}(-\delta(\m))  +O(p).
$$
\end{lemma}

Assuming this to be true for the moment we can  establish 
Lemmas \ref{lem:need1} and \ref{lem:need2}.	
Beginning with the latter, 
we  deduce
that 
$|\cD_{d'}^*(\m')|\leq
4^{\omega(d')}\tau(d')d'^{2} (d',\delta(\m'))$. 
Once inserted into
\eqref{eq:seek}, we obtain
$$
|\mathcal D_{d}(\ma{m})|
\leq 4^{\omega(d)}\tau(d)d^{2}
\sum_{\substack{h\mid (d,  \m)}} h^{2}
\left(\frac{d}{h}, \frac{\delta(\m)}{h^4}\right)
\leq 4^{\omega(d)}\tau(d)^2d^{2}
(d,\m)
(d, \delta(\m)),
$$
since $(d,\Delta_V)=1$. This therefore establishes Lemma \ref{lem:need2}.

We now turn to the proof of Lemma \ref{lem:need1}, 
and recall the definition of $\Sigma(x)$ in 
\eqref{eq:need1-pre}.
Suppose first that  $\delta(\m)\neq 0$. Then it follows from Lemma \ref{lem:need2} that 
$\cD_d(\m)\ll d^{2+\ve}$,
since  $\gcd(d,\delta(\m))=\gcd(d,\m)=1$. But then $\Sigma(x)\ll x^{3+\ve}$. 
If $\delta(\m)=0$ then Lemma \ref{lem:need2} yields the trivial bound
$$
\Sigma(x)\ll \sum_{d\leq x} d^{3+\ve} (d,\m)\ll |\m|^\ve x^{4+\ve},
$$
since $\m\neq \0$.

When $\delta(\m)H(\m)\neq 0$ we can do better using 
complex analysis. Recall the definition \eqref{eq:N_T2} of $N$.
Let $M$ be any non-zero integer divisible by $qN$. We will need to
examine the Dirichlet series 
$$
\eta_M(s;\m)=
\sum_{(d,M)=1} \frac{\chi(d)\cD_{d}(\m)}{d^s}=
\prod_{p\nmid M}\left\{\sum_{r=0}^\infty
\frac{\chi_{p^r}(-1)\cD_{p^r}(\m)}{p^{rs}}\right\},
$$
for $s \in \CC$.
One deduces from   \eqref{eq:seek} and Lemmas~\ref{lem:need2} and \ref{lem:mawkish} 
 that
\begin{align*}
\sum_{r=0}^\infty
\frac{\chi_{p^r}(-1)\cD_{p^r}(\m)}{p^{rs}}
&=
1+\frac{\chi_{p}(-1)\cD_{p}(\m)}{p^{s}}
+O\left(
\sum_{r=2}^\infty
4(r+1)^2p^{2r-r\sigma}\right)\\
&=
1+\frac{\chi_{p}(\delta(\m))}{p^{s-2}}
+O\left(p^{1-\sigma}+p^{4-2\sigma}\right),
\end{align*}
for the primes under consideration, 
where $\sigma=\Re(s)$. 
It  follows that 
$$\eta_M(s;\m)=L(s-2,\psi_\m)E_M(s),
$$ 
where $L(s,\psi_\m)$ is the Dirichlet $L$-function with Jacobi symbol $\psi_\m(\cdot) =(\frac{\delta(\m)}{\cdot})$ and
$E_M(s)$ is an Euler product which converges absolutely in the half
plane $\sigma>\frac{5}{2}$ and satisfies the bound $E_M(s)=O( M^\ve)$
there. 
One notes that $\psi_\m$ is non-trivial when $\delta(\m)\neq \square$ and 
has conductor $O(|\m|^4)$ since $\delta$ is a quartic form.

Invoking the truncated Perron formula, 
as in \cite[Eq.\ (4.5)]{BM} with $c=3+\ve$, we easily arrive at the statement
of Lemma \ref{lem:need1} when $\delta(\m)\neq \square$ by repeating the
proof of \cite[Lemma~18]{BM} and moving the line of integration back to
$\sigma=\frac{5}{2}+\ve$. 
When $\delta(\m)=\square$, we see that 
$\eta_M(s;\m)$ is absolutely convergent and bounded by $O(M^\ve)$ in the
half-plane $\sigma>3$.
We apply the  truncated Perron formula, 
as in \cite[Eq.\ (4.5)]{BM} , with $c=3+\ve$. But then we immediately  arrive at the second 
estimate in Lemma \ref{lem:need1} without
the need to move the line of integration.

\begin{proof}[Proof of Lemma \ref{lem:mawkish}]
In order to apply \eqref{eq:gauss} we  must deal
with the possibility that one of the linear forms $L_{i}(\b)$ is
divisible by $p$.  Let us write 
\begin{equation}\label{eq:sum_j}
\cD_{p^{r}}^*(\m) = \sum_{0\leq j<r} \cD_{p^{r}}^{(j)}(\m) +
\cD_{p^{r}}^{(r)}(\m),
\end{equation}
where for $0\leq j<r$ we denote by $\cD_{p^{r}}^{(j)}(\m) $ 
the overall contribution to $\cD_{p^r}^*(\m)$ from $\b$ such that  
$p^{j}\| g(\b)$. Likewise $\cD_{p^{r}}^{(r)}(\m) $ is the contribution from $\b$ such that
$p^{r}\mid g(\b)$. 
For the first part of the lemma it will suffice to show that 
\begin{equation}\label{eq:bang}
\cD_{p^{r}}^{(j)}(\m)
\leq 4p^{2r+\min\{r,\frac{v_p(\delta(\m))}{2}\}},
\end{equation}
for $0\leq j\leq r$.

We begin with an analysis of $\cD_{p^{r}}^{(0)}(\m)$, 
recalling the notation \eqref{eq:psy} for $\cD_{p^{r}}(\m;\b)$. 
Note that
$p\nmid b_{2}$ in $\cD_{p^{r}}(\m;\b)$. Applying \eqref{eq:gauss}, we deduce that
\begin{align*}
\cD_{p^{r}}(\m;\b) 
&= p^{3r}\chi_{p^{r}}
\left(-g(\b\right)^{2})e_{p^{r}}\left(-\overline{4}\left\{ \sum_{i=1}^{6} \overline{L_{i}(\b)} m_{i}^{2}\right\} \right)\\
&= p^{3r}\chi_{p^{r}}(-1) e_{p^{r}}\left(-\overline{4b_{2}}\left\{ \sum_{i=1}^{6} \overline{L_{i}
(b_{1}\overline{b_{2}},1)} m_{i}^{2}\right\} \right).
\end{align*}
Dividing by $p^{2r}$, 
introducing the sum over $\b$ and making the change of variables $b=b_{1}\overline{b_{2}}$, 
we obtain
\begin{equation}\label{eq:D0}
\cD_{p^{r}}^{(0)}(\m)=p^{r} \chi_{p^{r}}(-1) \sum_{\substack{b \bmod{p^{r}}\\ 
p\nmid g(b,1)}}
c_{p^{r}}\left(q_{\m}(b)\right),
\end{equation}
where
\begin{equation}\label{eq:quid}
\begin{split}
q_{\m}(b)
&=\be''L'(b,1)(m_1^2+m_2^2)+\be''L(b,1)(m_3^2+m_4^2)+
L(b,1)L'(b,1)(m_5^2+m_6^2)\\
&=c_0b^2+c_1b+c_2,
\end{split}
\end{equation}
in the notation of \eqref{eq:ci}. Recall from \eqref{eq:delta-given} that  $\delta(\m)=c_1^2-4c_0c_2$.
It follows from \eqref{eq:ram} and Lemma \ref{lem:rho} that 
$$
|\cD_{p^{r}}^{(0)}(\m)|
\leq p^{2r}  \left(\rho_{q_\m}(p^r)+\rho_{q_\m}(p^{r-1})\right)
\leq 4p^{2r+\min\{r,\frac{v_p(\delta(\m))}{2}\}}.
$$
This is satisfactory for \eqref{eq:bang}.

Next we require an estimate of similar strength for the sums 
$\cD_{p^{r}}^{(j)}(\m)$, for $1\leq j<r$.
We begin by noting that if $p\mid g(\b)$ in $\cD_{p^{r}}(\m;\b)$ then $p$ can divide 
precisely one of the linear factors since $p\nmid \Delta_V \b$.   Consequently we may write 
$\cD_{p^{r}}^{(j)}(\m)=D+D'+D''$ where $D$ denotes the overall  contribution arising from $\b$ such that 
$p^j\| L(\b)$, and similarly for $D'$ and $D''$. We will
 deal here only with the  sum $D$, which is typical.
Since $p\nmid L'(\b)L''(\b)$ we deduce that $p\nmid b_2$ and \eqref{eq:gauss} yields
$$
\cD_{p^r}(\m;\b) = p^{2r} 
e_{p^{r}}\left(-\overline{4b_2}\left\{ \sum_{i=3}^{6} \overline{L_{i}(b_1\bar{b_2},1)} m_{i}^{2}\right\} \right) \mathcal{G}_1\mathcal{G}_2,
$$
where
$$
\mathcal{G}_i=\sum_{k\bmod{p^r}} e_{p^r}(L(\b)k^2+m_i k),
$$
for $i=1,2.$   Write $L(b_1\bar{b_2},1)=p^j u$, with $p\nmid u$. It is easily checked that $\mathcal{G}_i=0$ unless $p^j\mid m_i$, for $i=1,2$. Writing $m_i=p^jm_i'$ for $i=1,2$, 
 it follows from a further application of \eqref{eq:gauss} that 
\begin{align*}
\mathcal{G}_i
&=p^j
\sum_{k\bmod{p^{r-j}}} e_{p^{r-j}}(b_2uk^2+m_i' k)\\
&=
\begin{cases}
p^{\frac{r+j}{2}} 
e_{p^{r-j}}(-\overline{4 b_2 u}m_{i}'^{2}), & \mbox{if $r-j$ is even,}\\
\ve(p) \chi_p(b_2u)
p^{\frac{r+j}{2}} 
e_{p^{r-j}}(-\overline{4b_2 u}m_{i}'^{2}), & \mbox{if $r-j$ is odd.}
\end{cases}
\end{align*}
Hence
$$
|D|\leq \frac{1}{p^{2r}} \cdot 
p^{3r+j} 
\left|
\sum_{\substack{\b \bmod{p^r}\\
L(\b)=p^j u\\
p\nmid b_2u
 }}
e_{p^{r}}\left(-\overline{4b_2}\left\{ \sum_{i=3}^{6} \overline{L_i(b_1\bar{b_2},1)} m_{i}^{2}\right\} \right) 
e_{p^{r-j}}\left(-\overline{4b_2u}(m_{1}'^{2}+m_2'^2)\right)\right|
$$
if $p^j\mid (m_1,m_2)$ and $D=0$ otherwise.  
We proceed under the assumption that $p^j\mid (m_1,m_2)$.
We may view the sum over $b_2$ as a Ramanujan sum, leading to the inequality
$$
|D|
\leq 
p^{r+j} 
\sum_{\substack{b \bmod{p^r}\\
L(b,1)\equiv 0 \bmod{p^j}
 }}
|c_{p^{r}}\left(r_{\m}(b)\right) |,
$$
where $r_{\m}(b)=p^{-j}q_{\m}(b)$, in the notation of \eqref{eq:quid}.

We handle the remaining sum by writing $b=b_0+p^jb'$ where
$b_0\equiv-\bar{\alpha}\beta\bmod{p^j}$ and $b'$ runs modulo $p^{r-j}$.  
Let 
\begin{align*}
r_\m'(x)=r_\m(b_0+p^jx)
&=p^{-j}q_\m(b_0+p^jx)\\
&=p^jc_0 x^2+(2b_0c_0+c_1)x+p^{-j}q_\m(b_0),
\end{align*}
and note that $r_\m'$ has discriminant $c_1^2-4c_0c_2=\delta(\m)$. 
We have 
$$
|D|
\leq 
p^{r+j} 
\sum_{b'\bmod{p^{r-j}}}
|c_{p^r}(r'_\m(b'))|.
$$
The summand is zero unless $p^{r-1}\mid r'_\m(b')$. In particular we may assume that $p^j\mid r'_{\m}(b')$, whence 
$|c_{p^r}(r'_\m(b'))|=p^j|c_{p^{r-j}}(p^{-j}r'_\m(b'))|$.
Hence it follows from 
 \eqref{eq:ram} and Lemma \ref{lem:rho} that
\begin{align*}
|D|
&\leq p^{2r+j} 
\left(\rho_{p^{-j}r_\m'}(p^{r-j})+\rho_{p^{-j}r_\m'}(p^{r-j-1})\right)
\leq 4p^{2r+\min\{r,\frac{v_p(\delta(\m))}{2}\}}.
\end{align*}
The same bound holds for $D',D''$ and so also for $\mathcal{D}_{p^r}^{(j)}(\m)$, as required for 
\eqref{eq:bang}.

Finally we consider $\cD_{p^r}^{(r)}(\m)$.  We adapt the preceding argument, dealing with the case $p^r\mid L(\b)$ and $p\nmid L'(\b)L''(\b)$, corresponding to $D$, say.
Tracing through the argument one deduces that 
$\mathcal{G}_1\mathcal{G}_2=0$ unless $p^r\mid (m_1,m_2)$, which we henceforth assume, in which case 
$\mathcal{G}_1\mathcal{G}_2=p^{2r}$.  Hence
\begin{align*}
D
&=p^{2r} 
\sum_{\substack{b\bmod{p^r}\\ L(b,1)\equiv 0 \bmod{p^r}}}
 c_{p^r}\left( \beta'' (m_3^2+m_4^2)+L'(b,1)(m_5^2+m_6^2)\right)\\
 &=p^{2r} 
 c_{p^r}\left( \alpha\beta'' (m_3^2+m_4^2)+(\alpha\beta'-\alpha'\beta)(m_5^2+m_6^2)\right)\\
 &=p^{2r} 
 c_{p^r}\left(\sigma(\m)\right),
\end{align*}
in the notation of \eqref{eq:sigma}.
We claim that 
$$
|c_{p^r}(\sigma(\m))|\leq p^{\min\{r,\frac{v_p(\delta(\m))}{2}\}}.
$$
Once achieved, and coupled with companion estimates for $D'$ and $D''$, this will
suffice for \eqref{eq:bang}.
Suppose first that $p^r\mid \sigma(\m)$, so that $c_{p^r}(\sigma(\m))=\phi(p^r)$. 
It will be convenient to write $\xi_{i}=m_i^2+m_{i+1}^2$, for $i\in \{1,3,5\}$.
Recalling that $p^{2r}\mid \xi_1$ and noting from 
\eqref{eq:ci} and \eqref{eq:sigma} that 
$\sigma(\m)=c_1-2\alpha'\beta \xi_5$, we see that 
\begin{align*}
\delta(\m)
&\equiv (\sigma(\m)
+2\alpha'\beta\xi_5)^2-4\alpha\alpha'\beta \xi_5
(\beta''\xi_3+\beta'\xi_5)
 \bmod{p^{2r}}\\
&\equiv 4\alpha'\beta\xi_5\left\{ \sigma(\m)
+(\alpha'\beta-\alpha\beta')\xi_5 -\alpha\beta''\xi_3\right\}
 \bmod{p^{2r}}\\
&\equiv 0 \bmod{p^{2r}}.
\end{align*}  
Hence $v_p(\delta(\m))\geq 2r$, as required. 
The case in which $p^{r-1}\| \sigma(\m)$ is  similar, since then 
$c_{p^r}(\sigma(\m))=-p^{r-1}$ and the same argument shows that $v_p(\delta(\m))\geq 2(r-1)$.

\medskip
It remains to analyse the case $r=1$ when $p\nmid H(\m)$. In this case
$$
\mathcal{D}_{p}^*(\m)=
\mathcal{D}_{p}^{(0)}(\m)+
\mathcal{D}_{p}^{(1)}(\m),
$$
by \eqref{eq:sum_j}.  Moreover, 
$\mathcal{D}_{p}^{(1)}(\m)=0$ since $p\nmid H(\m)$.
Taking $r=1$ in \eqref{eq:D0} we deduce from \eqref{eq:ram} that
$$
\cD_{p}^*(\m)=p \chi_{p}(-1) 
\left(
p
\sum_{\substack{b \bmod{p}\\ 
p\nmid g(b,1)\\
p\mid q_{\m}(b) }}1 -
\sum_{\substack{b \bmod{p}\\ 
p\nmid g(b,1)}}1
\right).
$$
The second sum in the brackets is $p+O(1)$ and the first sum can be written
$$
\rho_{q_{\m}}(p) - \#\{b\bmod{p}: p\mid g(b,1), ~p\mid q_{\m}(b)\}=\rho_{q_\m}(p),
$$
since $p\nmid H(\m)$.
Substituting the identity  $\rho_{q_\m}(p)=1+\chi_p(\delta(\m))$, we easily arrive at the statement of the lemma.
\end{proof}

\section{Analysis of $\mathcal{M}_{d,q}(\m)$}\label{s:last'}

In this section we establish the estimate in Lemma
\ref{lem:need3} for the mixed sum. Let $\m\in \ZZ^6$ and let $d,q \in \NN$ with $d\mid q^\infty$ and
$q\mid d^\infty$, such that $(d,\Delta_V)=1$.
We claim that 
\begin{equation}\label{eq:onion}
\cM_{d,q}(\m)
=\frac{q^6}{d^2}
\sum_{h\mid d}h^6
\sum_{\substack{r\mid q}}
 \frac{\mu(q/r)}{r^6}
\sum_{\substack{u\mid r\\ 
(uhq/r)\mid \m}}
u^6
 \sum_{\substack{b_1 \bmod{d'}\\b_2 \bmod{d'r'}\\
 (b_1,ub_2,d')=(b_2,r')=1 }}
\cD_{d'r'}\left(\m'';r'b_1,b_2\right),
\end{equation}
where 
$$
d'=d/h, \quad 
\m''=(uhq/r)^{-1}\ma{m}\in \ZZ^6, \quad r'=r/u.
$$
To see this we treat the sum over $a$ as a Ramanujan sum in \eqref{eq:S'}, finding that
\begin{align*}		
\cM_{d,q}(\m)
&=
\sum_{\substack{\mathbf k \bmod{dq}\\
Q_1(\k)\equiv 0\bmod{d}\\Q_2(\k)\equiv 0\bmod{d}}}
e_{dq}\left(\m.\k\right)
\sideset{}{^{*}}\sum_{a\bmod{q}}
e_{dq}\left(aQ_2(\k)\right)\\
&=
\sum_{r\mid q} r\mu\left(\frac{q}{r}\right)
\sum_{\substack{\mathbf k \bmod{dq}\\
Q_1(\k)\equiv 0\bmod{d}\\Q_2(\k)\equiv 0\bmod{dr}}}
e_{dq}\left(\m.\k\right).
\end{align*}
Breaking the inner sum into residue classes modulo $dr$ we easily deduce that $\m$ must  be divisible by $q/r$, whence 
$$
\cM_{d,q}(\m)
=\sum_{\substack{r\mid q\\ (q/r)\mid \m}} r\mu\left(\frac{q}{r}\right) \left(\frac{q}{r}\right)^6 
\sum_{\substack{\mathbf k \bmod{dr}\\
Q_1(\k)\equiv 0\bmod{d}\\Q_2(\k)\equiv 0\bmod{dr}}}
e_{dr}\left(\m'.\k\right),
$$
where $\m=(q/r)\m'$.
Using characters to detect the congruences involving $Q_1$ and $Q_2$, we may write the sum over $\k$ as
$$
\frac{1}{d^2r}
\sum_{b_1\bmod{d}}
\sum_{b_2\bmod{dr}}
\mathcal{D}_{dr}(\m';rb_1,b_2),
$$
in the notation of \eqref{eq:psy}.
Next we extract the greatest common divisor $h$ of $\b$ and $d$, writing $d=hd'$ and $\b=h\b'$, with $(\b',d')=1$.
Breaking the sum into congruence classes modulo $d'r$ we then see that 
\begin{align*}
\mathcal{D}_{dr}(\m';rb_1,b_2)
&=
\hspace{-0.1cm}
\sum_{\k' \bmod{d'r}} \sum_{\k'' \bmod{h}} 
\hspace{-0.2cm}
e_{d'r}\left(rb_1'Q_2(\k')+b_2'Q_2(\k')+h^{-1}\m'.\k'\right) e_h(\m'.\k'').
\end{align*}
This is 
$h^6 \mathcal{D}_{d'r}(h^{-1}\m';rb_1',b_2')$  if $h\mid \m'$ and 
$0$ otherwise.
Our work so far has shown that 
\begin{equation}\label{eq:train-2}
\cM_{d,q}(\m)
=\frac{q^6}{d^2}
\sum_{h\mid d}h^6
\sum_{\substack{r\mid q\\ 
(hq/r)\mid \m}}
 \frac{\mu(q/r)}{r^6}
 \sum_{\substack{b_1 \bmod{d'}\\b_2 \bmod{d'r}\\
 (\b,d')=1
 }}
\cD_{d'r}\left(h^{-1}\m';rb_1,b_2\right).
\end{equation}
To  complete the proof of \eqref{eq:onion}, we will need to 
extract the greatest  common divisor $u$ of  $b_2$ and
$r$, writing $b_2=ub_2'$ and $r=ur'$ with $(b_2',r')=1$. 
Breaking the sum in 
$\cD_{d'r}\left(\m';rb_1,b_2\right)
$ into residue classes modulo
$d'r'$, as before,
we conclude that the inner sum over $\b$ can be written
$$
\sum_{\substack{u\mid r\\ u \mid h^{-1}\m'}} u^6
 \sum_{\substack{b_1 \bmod{d'}\\b_2' \bmod{d'r'}\\
 (b_1,ub_2',d')=(b_2',r')=1 }}
\cD_{d'r'}\left(\m'';r'b_1,b_2'\right),
$$
where $\m''=u^{-1}\m'$. This concludes the proof of \eqref{eq:onion}.

Returning to \eqref{eq:onion}, we
denote by $D(d',r')$ the inner sum over $\b=(b_1,b_2)$. 
We will establish the following estimate

\begin{lemma}\label{lem:ghoul}
Let $d',r'\in \NN$, with $d'\mid {r'}^\infty$ and $(d',\Delta_V)=1$. 
Then we have 
$$
D(d',r')
\ll
d'^{4+\ve} {r'}^{3+\ve}
(d',\delta(\m'')) (r',Q_2^*(\m'')).
$$
\end{lemma}

Inserting this estimate into \eqref{eq:onion}, 
we see that the contribution to $\cM_{d,q}(\m)$ from $h\neq d$
is 
\begin{align*}
&\ll(dq)^\ve
\frac{q^6}{d^2}
\sum_{\substack{h\mid d\\ h\neq d}}h^5
\sum_{\substack{r\mid q}}
 \frac{1}{r^6}
\sum_{\substack{u\mid r\\ 
(uhq/r)\mid \m}}
u^5
\left(\frac{d}{h}\right)^{4} \left(\frac{r}{u}\right)^{3}
(d,\delta(\m)) (r,Q_2^*((q/r)^{-1}\m))\\
&\ll d^{2+\ve}q^{3+\ve}
\sum_{\substack{h\mid d\\ h\neq d}}h 
\sum_{\substack{r\mid q}}
 \left(\frac{q}{r}\right)^2
\sum_{\substack{u\mid r\\ 
(uhq/r)\mid \m}}u^2
(d,\delta(\m)) (q,Q_2^*(\m))\\
&\ll d^{2+\ve}q^{3+\ve}
 (d,\m)(q,\m)^2
(d,\delta(\m)) (q,Q_2^*(\m)).
\end{align*}
This is satisfactory for Lemma \ref{lem:need3}.

Turning to the contribution to 
$\cM_{d,q}(\m)$ from the terms with $h=d$, we see from \eqref{eq:train-2} that this is equal to
$$
d^4q^6
\sum_{\substack{r\mid q\\ 
(dq/r)\mid \m}}
\frac{ \mu(q/r)}{r^6}
\sum_{\substack{b_2 \bmod{r}}}
\cD_{r}\left(\m';0,b_2\right),
$$
with $\m'=(dq/r)^{-1}\ma{m}\in \ZZ^6$.
A little thought reveals that 
$$
\sum_{\substack{b_2 \bmod{r}}}
\cD_{r}\left(\m';0,b_2\right)=
\sum_{\substack{h\mid r\\ h\mid \m'}} h^6 \cQ_{r'}(\m''),
$$
with $r=hr'$ and $\m'=h\m''$.  But then it follows from \cite[Lemma
15]{BM} that 
\begin{align*}
\sum_{\substack{b_2 \bmod{r}}}
\cD_{r}\left(\m';0,b_2\right)
&\ll 
\sum_{\substack{h\mid r\\ h\mid \m'}} h^6 r'^{3} (r',Q_2^*(\m''))\\
&\ll r^{3+\ve} (r,\m')^2(r,Q_2^*(\m')).
\end{align*}
Since $d\mid \m$ and $\delta(\m)$ is homogeneous, we clearly have $d^2\leq (d,\m)(d,\delta(\m))$. Hence this case  contributes 
\begin{align*}
&\ll  
d^2 q^{3+\ve} 
(d,\m)
(q,\m)^2(d,\delta(\m)) (q,Q_2^*(\m))
\end{align*}
to $\cM_{d,q}(\m)$.
This too is satisfactory and so completes the proof of Lemma 
\ref{lem:need3}  subject  
to the verification of Lemma \ref{lem:ghoul}.

\begin{proof}[Proof of Lemma \ref{lem:ghoul}]
We recall that
$$
D(d',r')=
 \sum_{\substack{b_1 \bmod{d'}\\b_2 \bmod{d'r'}\\
 (b_1,ub_2,d')=(b_2,r')=1 }}
\cD_{d'r'}\left(\m'';r'b_1,b_2\right),
$$
This sum satisfies a basic multiplicativity 
property, meaning that it will suffice to analyse the case in which 
$d'=p^k$ and $r'=p^\ell$ for integers $k\geq 1$ and $\ell\geq 0$, with 
 $p\nmid \Delta_V$.
Let us write $u=p^j$. 

Suppose first that  $\ell=0$ and $k\geq 1$.  
Then 
$$
D(d',r')=
D(p^k,1)=
 \sum_{\substack{b_1 \bmod{d'}\\b_2 \bmod{d'}\\
 (b_1,ub_2,d')=1 }}
\cD_{d'}\left(\m'';b_1,b_2\right),
$$
with $d'=p^k$.
This is equal to a modified version of
$d'^2\cD_{d'}^*(\m'')$, in the notation of \eqref{eq:seek},
wherein one is only interested in $\b$ for which
$(b_1,ub_2,d')=1$. 
Obviously this precisely coincides with $d'^2\cD_{d'}^*(\m'')$ when $j=0$.
The proof of
Lemma \ref{lem:mawkish} therefore
 shows that 
\begin{equation}\label{eq:one}
|D(p^k,1)|\leq 
4(k+1)p^{4k+\min\{k,\frac{v_p(\delta(\m''))}{2}\}},
\end{equation}
where $\delta(\m'')$ is given by \eqref{eq:delta-given}.

Next suppose that $\ell\geq 1$ and $k\geq 1$. In this case
$$
D(d',r')=
D(p^k,p^\ell)=
 \sum_{\substack{b_1 \bmod{p^k}\\b_2 \bmod{p^{k+\ell}}\\
 p\nmid b_1b_2}}
\cD_{p^{k+\ell}}\left(\m'';p^\ell b_1,b_2\right).
$$
Recall
 the notation \eqref{eq:LL'L''} for $L,L',L''$ and the
subsequent definition of $g$.  It is clear that 
$p\nmid g(p^{\ell}b_1,b_2)$, since $p\nmid b_2$. We may now trace
through the analysis leading to \eqref{eq:D0}, finding that 
\begin{align*}
|D(p^k,p^\ell)|
&\leq  p^{3(k+\ell)}
\sum_{\substack{b \bmod{p^k}}}
\left|c_{p^{k+\ell}}\left(q_{\m''}(p^\ell b)\right)\right|\\
&\leq  p^{4(k+\ell)}\left(
\rho_{r_{\m''}}(p^{k})+
\rho_{r_{\m''}}(p^{k-1})\right),
\end{align*}
via \eqref{eq:ram}, 
where $r_{\m''}(x)=p^{-\ell} q_{\m''}(p^\ell x)$.
Here we have observed that the right hand side is empty unless $p^{\ell}\mid c_2=Q_2^*(\m'')$. But the discriminant of $r_{\m''}$ is equal to $\delta(\m'')$. Hence Lemma \ref{lem:rho}
yields
\begin{align*}
|D(p^k,p^\ell)|
&\leq  4p^{4(k+\ell)+\min\{k,\frac{v_p(\delta(\m''))}{2}\}}.
\end{align*}
Combining this with \eqref{eq:one},  we readily arrive at the statement
of Lemma \ref{lem:ghoul}.
\end{proof}

\end{document}